\documentclass[a4paper,11pt]{amsart}

\numberwithin{equation}{section}
\usepackage[top=3.5cm,bottom=3.5cm,left=3.6cm,right=3.6cm]{geometry}
\usepackage{amsfonts}
\usepackage{amsmath}
\usepackage{amssymb}
\usepackage{amsthm}
\usepackage[T1]{fontenc}
\usepackage{graphicx}
\usepackage{hyperref}
\usepackage{enumerate}%pour changer la numérotation
\usepackage{mathtools}
\usepackage{enumitem}

\newtheorem{dfn}{Definition}[section]
\newtheorem{propo}[dfn]{Proposition}
\newtheorem{theo}[dfn]{Theorem}

\newtheorem{cor}[dfn]{Corollary}
\newtheorem{lem}[dfn]{Lemma}

\theoremstyle{definition}
\newtheorem{nota}[dfn]{Notation}
\theoremstyle{remark}

\newtheorem{rem}[dfn]{Remark}

\newcommand{\norm}[1]{\left\lVert #1 \right\rVert}
\newcommand{\ens}[1]{\left\{ #1\right\}}
\newcommand{\R}{\mathbb{R}}
\newcommand{\Z}{\mathbb{Z}}

\newcommand{\abs}[1]{\left|#1\right|}

\newcommand{\eq}[1]{\begin{equation}#1\end{equation}}

\providecommand{\newoperator}[3]{%
  \newcommand*{#1}{\mathop{#2}#3}}
\newoperator{\re}{\mathrm{Re}}{\,}
\newoperator{\im}{\mathrm{Im}}{\,}

\begin{document}

\title{Holderian weak invariance principle under a Hannan type condition}

\author{Davide Giraudo}

\address{Universit\'e de Rouen, LMRS, Avenue de l'Universit\'e, BP 12 76801 
Saint-\'Etienne-du-Rouvray cedex, France.}

\email{davide.giraudo1@univ-rouen.fr}

\date{\today}

\keywords{Invariance principle, martingales, strictly stationary process}

\subjclass[2010]{60F05; 60F17}

 \begin{abstract}
  We investigate the invariance principle in H\"older spaces for strictly stationary 
  martingale difference sequences. In particular, we show that the sufficient condition on 
  the tail in the i.i.d. case does not extend to stationary ergodic martingale 
  differences. We provide a sufficient condition on the conditional variance 
  which guarantee the invariance principle in H\"older spaces. We then 
  deduce a condition in the spirit of Hannan one.
 \end{abstract}
 
 \maketitle

 \section{Introduction}
 
 One of the main problems in probability theory is the understanding of the 
 asymptotic behavior of Birkhoff sums $S_n(f):=\sum_{j=0}^{n-1}f\circ T^i$, where $(\Omega,
 \mathcal F,\mu,T)$ 
 is a dynamical system and $f$ a map from $\Omega$ to the real line. 
 
 One can consider random functions contructed from the Birkhoff sums
 \begin{equation}\label{polyg_f}
   S_n^{\mathrm{pl}}(f,t):=S_{[nt]}(f)+ (nt-[nt])f\circ T^{[nt]+1},\quad t\in [0,1].
   \end{equation}
and investigate the asymptotic behaviour of the sequence $\left(S_n^{\mathrm{pl}}(f,t)
\right)_{n\geqslant 1}$ seen as an element of a function space. 
 Donsker showed (cf. \cite{MR0040613}) that the sequence
  $(n^{-1/2}(\mathbb E(f^2))^{-1/2}S_n^{\mathrm{pl}}(f))_{n\geqslant 1}$ 
  converges in distribution in the 
  space of continuous functions on the unit interval to a standard Brownian 
  motion $W$ when the sequence $(f\circ T^i)_{i\geqslant 0}$ is i.i.d. and 
  zero mean. Then an intensive research has then been performed to extend this result to 
  stationary weakly dependent  sequences. 
  We refer the reader to \cite{MR2206313} for the main theorems in this direction.
  
  Our purpose is to investigate the weak convergence of the sequence 
  $(n^{-1/2}S_n^{\mathrm{pl}}(f))_{n\geqslant 1}$  in 
  H\"older spaces when $(f\circ T^i)_{i\geqslant 0}$ is a strictly stationary sequence. 
  A classical method for showing a limit theorem is to use a martingale approximation, which 
  allows to deduce the corresponding result if it holds for martingale differences sequences 
  provided that the approximation is good enough. To the best of our knowledge, no 
  result about the invariance principle in H\"older space for stationary martingale difference 
  sequences is known. 
  
 \subsection{The H\"older spaces}
 
  It is well known that standard Brownian motion's paths are almost surely 
  H\"older regular of exponent 
  $\alpha$ for each $\alpha\in (0,1/2)$, hence it is natural to consider the 
  random function defined in 
  \eqref{polyg_f} as an element of $\mathcal H_\alpha[0,1]$ and 
  try to establish its weak 
  convergence to a standard Brownian motion in this function space. 
  
  Before stating the results in this direction, let us 
  define for $\alpha\in (0,1)$ the H\"older space $\mathcal H_\alpha[0,1]$ of functions 
  $x\colon[0,1]\to\R$
  such that $\sup_{s\neq t}\abs{x(s)-x(t)}/\abs{s-t}^\alpha$ is finite. 
  The analogue of the continuity modulus in $C[0,1]$ is $w_\alpha$, defined by 
  \begin{equation}
  w_\alpha(x,\delta)=\sup_{0<\abs{t-s}<\delta}\frac{\abs{x(t)-x(s)}}{\abs{t-
  s}^\alpha}.
  \end{equation}
  We then define $\mathcal H_\alpha^0[0,1]$ by $\mathcal H_\alpha^0[0,1]:=
  \ens{x\in \mathcal H_\alpha[0,1],
  \lim_{\delta\to 
  0}w_\alpha(x,\delta)=0}$. We shall essentially work with the space 
  $\mathcal H_\alpha^0[0,1]$ which, 
  endowed with $\norm{x}_\alpha:=w_\alpha(x,1)+\abs{x(0)}$, is a separable Banach space 
  (while $\mathcal H_\alpha[0,1]$ is not separable). Since the canonical 
  embedding $\iota\colon \mathcal H^o_\alpha[0,1]\to \mathcal H_\alpha[0,1]$
  is continuous, each 
  convergence in distribution in $\mathcal H_\alpha^o[0,1]$ also takes place in 
  $\mathcal H_\alpha[0,1]$. 
  
  Let us denote by $D_j$ the set of dyadic numbers in $[0,1]$ of level $j$, that 
  is, 
  \begin{equation}
  D_0:=\ens{0,1},\quad D_j:=\ens{(2l-1)2^{-j};1\leqslant l\leqslant 2^{j-1}}, 
  j\geqslant 1.
  \end{equation}
If $r\in D_j$ for some $j\geqslant 0$, we define $r^+:=r+2^{-j}$ and 
$r^-:=r-2^{-j}$. 
For $r\in D_j$, $j\geqslant 1$, let $\Lambda_r$ be the function whose graph 
is the polygonal path joining the points $(0,0)$, $(r^-,0)$, $(r,1)$, 
$(r^+,0)$ and $(1,0)$. We can 
decompose each $x\in C[0,1]$ as 
\begin{equation}
x=\sum_{r\in D}\lambda_r(x)\Lambda_r=\sum_{j=0}^{+\infty}\sum_{r\in D_j}
\lambda_r(x)\Lambda_r,
\end{equation}
and the convergence is uniform  on $[0,1]$. The coefficients $\lambda_r(x)$ are given 
by 
\begin{equation}\label{eq:def_lambda_r}
\lambda_r(x)=x(r)-\frac{x(r^+)+x(r^-)}2, \quad r\in D_j,j\geqslant 1,
\end{equation}
and $\lambda_0(x)=x(0)$, $\lambda_1(x)=x(1)$. 

Ciesielski proved (cf. \cite{MR0132389}) that $\ens{\Lambda_r;r\in D}$ is a 
Schauder basis of 
$\mathcal H_\alpha^o[0,1]$ and the norms $\norm{\cdot}_\alpha$ and the sequential norm 
defined by
\begin{equation}
\norm{x}_\alpha^{\operatorname{seq}}:=\sup_{j\geqslant 0}2^{j\alpha}\max_{r\in 
D_j}\abs{\lambda_r(x)},
\end{equation}
are equivalent. 

Considering the sequential norm, we can show (see Theorem~3 in \cite{MR1736910}) that a sequence 
$(\xi_n)_{n\geqslant 1}$ of random elements of $\mathcal H_\alpha^o$ 
vanishing at $0$ is tight if and only if for each positive 
$\varepsilon$, 
\begin{equation}
\lim_{J\to\infty}\limsup_{n\to \infty}\mu\ens{\sup_{j\geqslant J}2^{j\alpha}
\max_{r\in D_j}\abs{\lambda_r(\xi_n)}>\varepsilon}=0.
\end{equation}

\begin{nota}\label{first_notation}
In the sequel, we will denote $r_{k,j}:=k2^{-j}$ and $u_{k,j}:= [nr_{k,j}]$ 
(or $r_k$ and $u_k$ for 
short). Notice that $u_{k+1,j}-u_{k,j}=
[nr_{k,j}+n2^{-j}]-u_{k,j}\leqslant 2n2^{-j}$ if $j\leqslant \log n$, where 
$\log n$ denotes 
the binary logarithm of $n$ and for a real number $x$, $[x]$ is the unique 
integer for which $[x]\leqslant x< [x]+1$. 
\end{nota}

\begin{rem}\label{rmq:simplification_sequential_norm}
Since for each $x\in\mathcal H_{1/2-1/p}[0,1]$, each $j\geqslant1$ 
and each $r\in D_j$, 
\begin{multline}
 \abs{\lambda_r(x)}\leqslant \frac{\abs{x(r^+)-x(r)}}2+
 \frac{\abs{x(r)-x(r^-)}}2\leqslant \\
 \leqslant\max\ens{\abs{x(r^+)-x(r)}, 
 \abs{x(r)-x(r^-)}},
\end{multline}
for a function $f$, the sequential norm of $n^{-1/2}S_n^{\mathrm{pl}}(f)$ does not exceed 
\begin{equation}
 \sup_{j\geqslant 1} 2^{\alpha j}n^{-1/2}
  \max_{0\leqslant k<2^j}\abs{S_n^{\mathrm{pl}}(f,r_{k+1,j})-S_n^{\mathrm{pl}}(f,r_{k,j})}.
\end{equation}

\end{rem}

Now, we state the result obtained by Ra\v{c}kauskas and Suquet in \cite{MR2000642}.

\begin{theo}\label{thm_R_S}
 Let $p>2$ and let $(f\circ T^j)_{j\geqslant 0}$ be an i.i.d. centered sequence with 
 unit variance. Then 
 the condition 
 \begin{equation}\label{condition_suff_iid}
  \lim_{t\to \infty}t^p\mu\ens{\abs f>t}=0
 \end{equation}
is equivalent to the weak convergence of 
the sequence $(n^{-1/2}S_n^{\mathrm{pl}}(f))_{n\geqslant 1}$ 
to a standard Brownian motion in the space $\mathcal H_{1/2-1/p}^o[0,1]$. 
\end{theo}

 \subsection{Some facts about the $\mathbb L^{p,\infty}$ spaces}
 
 In the rest of the paper, $\chi$ denotes the indicator function.
 Let $p>2$. We define the $\mathbb L^{p,\infty}$ space as the collection 
 of functions $f\colon\Omega\to \mathbb R$ such that the quantity 
 \begin{equation}
  \norm f_{p,\infty}^p:=\sup_{t>0}t^p\mu\ens{\abs f>t}<\infty.
  \end{equation}
 This quantity is denoted like a norm, while it is not a norm (the triangle inequality 
 may fail, for example if $X=[0,1]$ endowed with the Lebesgue measure, 
 $f(x):=x^{-1/p}$ and $g(x):=f(1-x)$; in this case $\norm{f+g}_{p,\infty}
 \geqslant 2^{1+1/p}$ but $\norm{f}_{p,\infty}+\norm{g}_{p,\infty}=2$). However, there 
 exists a 
 constant $\kappa_p$ such that for each $f$, 
 \begin{equation}\label{link_weak_lp_norms}
  \norm f_{p,\infty}\leqslant \sup_{A:\mu(A)>0}\mu(A)^{-1+1/p}
  \mathbb E[\abs f\chi_A]\leqslant \kappa_p\norm f_{p,\infty}
 \end{equation}
 and $N_p(f):=\sup_{A:\mu(A)>0}\mu(A)^{-1+1/p}
  \mathbb E[\abs f\chi_A]$ defines a norm.
 The first inequality in \eqref{link_weak_lp_norms} can be seen 
 from the estimate $t\mu\ens{\abs f>t}\leqslant 
 \mathbb E\left[\abs f\chi\ens{ \abs f>t}\right]$; for the second one, 
 we write 
 \begin{equation}
   \mathbb E[\abs f\chi_A]=\int_0^{+\infty}\mu\left(
   \ens{\abs f>t}\cap A\right)\mathrm dt
   \leqslant\int_0^{+\infty}\min\ens{\mu
   \ens{\abs f>t},\mu(A)}\mathrm dt,
 \end{equation}
 and we bound the integrand by $\min\ens{t^{-p}\norm{f}^p_{p,\infty},\mu(A)}$.
 
 A function $f$ satisfies \eqref{condition_suff_iid} if and only if it belongs 
 to the closure of bounded functions with respect to $N_p$. Indeed, if $f$ satisfies 
 \eqref{condition_suff_iid}, then the sequence $(f\chi\abs{f<n})_{n\geqslant 1}$ 
 converges to $f$ in $\mathbb L^{p,\infty}$. If $N_p(f-g)<\varepsilon$ with 
 $g$ bounded, then 
 \begin{equation}
  \limsup_{t\to \infty}t^p\mu\ens{\abs f>t}\leqslant 
  \limsup_{t\to \infty}t^p\mu\ens{\abs{f-g}>t/2} \leqslant 2^p\varepsilon.
 \end{equation}

 We now provide two technical lemmas about $\mathbb L^{p,\infty}$ spaces. The 
 first one will be used in the proof of the weak invariance principle 
 for martingales, since we will have to control the tail function 
 of the random variables involved in 
 the construction of the truncated martingale (cf.\eqref{eq:definition_m_R}). The second 
 one will provides an estimation of the $\mathbb L^{p,\infty}$ norm of a simple function, 
 which will be used in the proof of Theorem~\ref{counter_example}, since the function 
 $m$ is contructed as a series of simple functions.
 
 \begin{lem}\label{tail_cond_expect}
  If $\lim_{t\to \infty}t^p\mu\ens{\abs f>t}=0$, then for each sub-$\sigma$-algebra 
  $\mathcal A$, we have $\lim_{t\to \infty}t^p\mu\ens{\mathbb E[\abs f\mid \mathcal A]>t}=0$.
 \end{lem}
 \begin{proof}
 For simplicity, we assume that $f$ is non-negative.
  For a fixed $t$, the set $\ens{\mathbb E[f\mid \mathcal A]>t}$ belongs to 
  the $\sigma$-algebra $\mathcal A$, hence 
  \begin{equation}
   t\mu\ens{\mathbb E[f\mid \mathcal A]>t}\leqslant 
   \mathbb E\left[\mathbb E[f\mid\mathcal A]\chi\ens{\mathbb E[f\mid \mathcal A]>t}\right]
   =\mathbb E[f\chi\ens{\mathbb E[ f\mid \mathcal A]>t}].
  \end{equation}
 
 By definition of $N_p$, 
 \begin{equation}
  \mathbb E[f\chi\ens{\mathbb E[f\mid \mathcal A]>t}]\leqslant 
  N_p\left(f\chi\ens{\mathbb E[f\mid \mathcal A]>t}\right)\mu\ens{\mathbb E[f\mid 
  \mathcal A]>t}^{1-1/p},
 \end{equation}
 
 hence 
 \begin{equation}
  t^p\mu\ens{\mathbb E[ f\mid \mathcal A]>t}\leqslant 
  N_p\left(f\chi\ens{\mathbb E[ f\mid \mathcal A]>t}\right)^p.
 \end{equation}

 Notice that 
 \begin{equation}
 \forall s>0, \quad N_p\left(f\chi\ens{\mathbb E[f\mid \mathcal A]>t}\right)
  \leqslant s\mu\ens{\mathbb E[ f\mid \mathcal A]>t}^{1/p}+
  N_p\left(f\chi\ens{f >s}\right),
 \end{equation}
 
 hence 
 \begin{equation}
  \limsup_{t\to \infty}\mathbb E[f\chi\ens{\mathbb E[\abs f\mid \mathcal A]>t}]\leqslant
  N_p(f\chi\ens{f >s}\leqslant \kappa_p\sup_{x\geqslant s}x^p\mu\ens{f>x}.
 \end{equation}

 By the assumption on the function $f$, the right hand side goes to $0$ as $s$ goes to infinity, 
 which concludes the proof of the lemma. 
 \end{proof} 
 
 \begin{lem}\label{estimate_norm_simple_function}
  Let $f:=\sum_{i=0}^Na_i\chi(A_i)$, where the family $(A_i)_{i=0}^N$ is pairwise 
  disjoint and $0\leqslant a_N<\dots<a_0$. Then 
  \begin{equation}
   \norm f_{p,\infty}^p\leqslant \max_{0\leqslant j\leqslant N}a_j^p\sum_{i=0}^j
   \mu(A_i).
  \end{equation}
 \end{lem}
\begin{proof}
 We have the equality
 \begin{equation}
  \mu\ens{f>t}=
   \sum_{j=0}^N\chi_{(a_{j+1}, a_j]}(t)\sum_{i=0}^j\mu(A_i),
 \end{equation}
where $a_{N+1}:=0$, therefore 
\begin{equation}
 t^p\mu\ens{f>t}\leqslant \max_{0\leqslant j\leqslant N}a_j^p\sum_{i=0}^j\mu(A_i).
\end{equation}

\end{proof}

 \section{Main results}
 
The goal of the paper is to give a sharp sufficient condition on the moments
of a strictly stationary martingale difference sequence which guarantees the 
weak invariance principle in $\mathcal H_\alpha^o[0,1]$ for a fixed $\alpha$.
 
 We first show that Theorem~\ref{thm_R_S} does not extend to strictly 
 stationary ergodic martingale difference sequences, that is, sequences 
 of the form $(m\circ T^i)_{i\geqslant 0}$ such that $m$ is $\mathcal M$ 
 measurable and $\mathbb E[m\mid T\mathcal M]=0$ for some sub-$\sigma$-algebra 
 $\mathcal M$ satisfying $T\mathcal M\subset\mathcal M$. 
 
 An application of Kolmogorov's continuity criterion shows that if 
 $(m\circ T^i)_{i\geqslant 0}$ is a martingale difference sequence such 
 that $m\in\mathbb L^{p+\delta}$ for some positive $\delta$ and $p>2$, 
 then the partial sum process $(n^{-1/2}S_n^{\mathrm{pl}}(m))_{n\geqslant 1}$ 
 is tight in $\mathcal H^o_{1/2-1/p}[0,1]$
  (see \cite{MR1108512}).

 We provide a condition on the quadratic variance which improves the previous 
 approach (since the previous condition can be replaced by $m\in\mathbb L^p$). 
 Then using martingale approximation we can provide a Hannan type condition 
 which guarantees the weak invariance principle in $\mathcal H_\alpha^o[0,1]$.
 
 \begin{theo}\label{counter_example}
 Let $p>2$ and $(\Omega,\mathcal F,\mu, T)$ be a dynamical system with positive entropy. 
 There exists a function $m\colon \Omega\to \R$ and a $\sigma$-algebra 
 $\mathcal M$ for which $T\mathcal M\subset\mathcal M$ such that:
 \begin{itemize}
  \item the sequence $(m\circ T^i)_{i\geqslant 0}$ is a martingale difference sequence 
  with respect to the filtration $(T^{-i}\mathcal M)_{i\geqslant 0}$;
  \item the convergence $\lim_{t\to +\infty}t^p \mu \ens{\abs m> t}=0$ takes place;
  \item the sequence $(n^{-1/2}S_n^{\mathrm{pl}}(m))_{n\geqslant 1}$ is not 
  tight in $\mathcal H^o_{1/2-1/p}[0,1]$.
 \end{itemize}
 \end{theo}

 \begin{theo}\label{sufficient_cond_mart}
 Let $(\Omega,\mathcal F,\mu,T)$ be a dynamical system, $\mathcal M$ 
 a sub-$\sigma$-algebra of $\mathcal F$ such that $T\mathcal M\subset\mathcal M$ 
 and $\mathcal I$ the collection of sets $A\in \mathcal F$ such that $T^{-1}=A$. 
 
 Let $p>2$ and let $(m\circ T^j,T^{-i}\mathcal M)$ be a strictly stationary martingale 
  difference sequence. Assume that $t^p\mu\ens{\abs m>t}\to 0$ and $\mathbb E[m^2\mid T\mathcal M]
  \in\mathbb L^{p/2}$. 
  Then 
  \begin{equation}\label{Lamperti_WIP}
   n^{-1/2}S_n^{\mathrm{pl}}(m)\to \eta\cdot W\mbox{ in distribution in }
 \mathcal H_{1/2-1/p}^o[0,1],
  \end{equation}
   where the random variable $\eta$ is given by
 \begin{equation}\label{def_eta}
  \eta=\lim_{n\to \infty}\mathbb E[S_n^2\mid\mathcal I]/n\mbox{ in }\mathbb L^1
 \end{equation} 
 and $\eta$ is independent of the process $(W_t)_{t \in [0,1]}$.
 
 In particular, \eqref{Lamperti_WIP} takes place if $m$ belongs to $\mathbb L^p$.
 \end{theo}

 The key point of the proof of Theorem~\ref{sufficient_cond_mart} is an 
 inequality in the spirit of Doob's one, which gives 
 $n^{-1}\mathbb E\left[\max_{1\leqslant j\leqslant n}S_j(m)^2\right]
 \leqslant 2\mathbb E[m^2]$. It is used in order to establish tightness of
 the sequence $(n^{-1/2}S_n^{\mathrm{pl}}(m))_{n\geqslant 1}$ in the space $C[0,1]$.
 
 \begin{propo}\label{generalized_Doob}
 Let $p>2$. There exists a constant $C_p$ depending only on $p$ such that if 
 $(m\circ T^i)_{i\geqslant 1}$ is a martingale difference sequence, then 
 the following inequality holds:
  \begin{equation}
   \sup_{n\geqslant 1}\norm{\norm{n^{-1/2}S_n^{\mathrm{pl}}(m)}_{\mathcal H^o_{1/2-1/p}}}
   _{p,\infty}^p\leqslant C_p\left(\norm{m}_{p,\infty}^p+
   \mathbb E\left(\mathbb E[m^2\mid T\mathcal M]\right)^{p/2}\right).
  \end{equation}
 \end{propo}

\begin{rem}\label{rmq:counter_example}
As Theorem~\ref{counter_example} shows, the condition 
$\lim_{t\to\infty}t^p\mu\ens{\abs m>t}=0$ alone for martingale difference sequences 
is not sufficient to obtain the weak convergence of $n^{-1/2}S_n^{\mathrm{pl}}(m)$ 
in $\mathcal H_\alpha^o[0,1]$ for $\alpha=1/2-1/p$. For the constructed $m$ in 
Theorem~\ref{counter_example}, the quadratic 
variance is $\kappa m^2$ for some constant $\kappa$ and $m$ does 
not belong to the $\mathbb L^p$ space. 
\end{rem}

By Lemma~A.2 in \cite{MR3020943}, the H\"older norm of 
a polygonal line is reached at two vertices, hence, for a function $g$,
\begin{align}
 \norm{n^{-1/2}S_n^{\mathrm{pl}}(g-g\circ T)}_{\mathcal H^o_{1/2-1/p}}
& =n^{-1/p}\max_{1\leqslant i<j\leqslant n}\frac{\abs{g\circ T^j-g\circ T^i}}{(j-i)^{1/2-1/p}} 
\label{eq:Holder_norm_coboundary}\\
 &\leqslant 2n^{-1/p}\max_{1\leqslant j\leqslant n}\abs{g\circ T^j}.
\end{align}
As a consequence, if $g$ belongs to $\mathbb L^p$, then
the sequence $\left(\norm{n^{-1/2}S_n^{\mathrm{pl}}(g-g\circ T)}_{\mathcal H^o_{1/2-1/p}}
\right)_{n\geqslant 1}$
converges to $0$ in probability . Therefore, we 
can exploit a martingale-coboundary decomposition in $\mathbb L^p$.

\begin{cor}\label{martingale_coboundary}
 Let $p>2$ and let $f$ be an $\mathcal M$-measurable function which can be written as 
 \eq{
  f=m+g-g\circ T,
 }
 where $m,g\in\mathbb L^p$ and $(m\circ T^i)_{i\geqslant 0}$ is a martingale difference sequence 
 for the filtration $(T^{-i}\mathcal M)_{i\geqslant 0}$. 
 Then $n^{-1/2}S_n^{\mathrm{pl}}(f)\to \eta W$ in distribution 
 in $\mathcal H_{1/2-1/p}^o[0,1]$, where $\eta$ is given by \eqref{def_eta} 
 and independent of $W$.
\end{cor}

We define for a function $h$ the operators $\mathbb E_k(h):=\mathbb E[h\mid T^k\mathcal M]$ 
and $P_i(h):=\mathbb E_i(h)-\mathbb E_{i+1}(h)$. The condition $\sum_{i=0}^{\infty}
\norm{P_i(f)}_2$ was 
introduced by Hannan in \cite{MR0331683} in order to deduce a central limit theorem. 
It actually implies the weak invariance principle (see Corollary~2 in 
\cite{MR2359065}).

 \begin{theo}\label{Hannan_cond}
  Let $p>2$ and let $f$ be an $\mathcal M$-measurable function such that 
  \begin{equation}\label{regularity}
   \mathbb E\left[f\mid \bigcap_{i\in\Z}T^i\mathcal M\right]=0\mbox{ and } 
  \end{equation}
  \begin{equation}\label{Hannan_condition}
   \sum_{i\geqslant 0}\norm{P_i(f)}_p<\infty.
  \end{equation}
 Then $n^{-1/2}S_n^{\mathrm{pl}}(m)\to \eta W$ in distribution 
 in $\mathcal H_{1/2-1/p}^o[0,1]$, where $\eta$ is given by \eqref{def_eta} 
 and independent of $W$.
 \end{theo}

 \section{Proofs}
 
 \subsection{Proof of Theorem~\ref{counter_example}}
 
 We need a result about dynamical systems of positive entropy for the construction 
 of a counter-example. 
 
 \begin{lem}\label{lem:positive_entropy}
  Let $(\Omega,\mathcal A,\mu,T)$ be an ergodic probability measure preserving system of 
  positive entropy. There exists two $T$-invariant sub-$\sigma$-algebras $\mathcal B$ and 
  $\mathcal C$ of $\mathcal A$ and a function $g\colon\Omega\to \R$ such that:
  \begin{itemize}
   \item the $\sigma$-algebras $\mathcal B$ and $\mathcal C$ are independent;
   \item the function $g$ is $\mathcal B$-measurable, takes the values $-1$, $0$ 
   and $1$, has zero mean and the process $(g\circ T^n)_{n\in\Z}$ is independent;
   \item the dynamical system $(\Omega,\mathcal C,\mu,T)$ is aperiodic.
  \end{itemize}

 \end{lem}

 This is Lemma~3.8 from \cite{MR1856684}. 
 
 We consider the following four increasing sequences of positive integers  $(I_l)_{l\geqslant 1}$, 
 $(J_l)_{l\geqslant 1}$, 
 $(n_l)_{l\geqslant 1}$  and $(L_l)_{l\geqslant 1}$. 
 We define $k_l:=2^{I_l+J_l}$ and impose the conditions:
 \begin{align}
  &\sum_{l=1}^{\infty}\frac 1{L_l}<\infty;\label{lac0}\\
  &\lim_{l\to \infty} J_l\cdot \mu\ens{\abs{\mathcal N}
  \geqslant 4^{1/p}\frac{L_l}{\norm g_2}}=1\label{lac3};\\ 
  &\lim_{l\to \infty}J_l 2^{-I_l/2}=0;\label{lac2}\\
  &\lim_{l\to \infty}n_l\sum_{i>l}\frac{k_i}{n_i}=0;\label{lac1}\\
  &\mbox{for each }l, \quad \sum_{i=1}^{l-1} \frac{k_l}{L_i}\left(\frac{n_i}{2^{I_i}}\right)^{1/p}
  <\frac{n_l^{1/p}}2. \label{lac4}
 \end{align}
 Here $\mathcal N$ denotes a random variable whose distribution is standard normal. 
 Such sequences can be constructed as follows: first pick a sequence $(L_l)_{l\geqslant 1}$ 
 satisfying \eqref{lac0}, for example $L_l=l^2$. Then construct $(J_l)_{l\geqslant 1}$ such that 
 \eqref{lac3} holds. Once the sequence $(J_l)_{l\geqslant 1}$ is constructed, 
 define $\left(I_l\right)_{l\geqslant 1}$ satisfying \eqref{lac2}. Now the 
 sequence $(k_l)_{l\geqslant 1}$ is completely determined. Noticing that 
 \eqref{lac1} is satisfied if the series $\sum_l k_ln_{l-1}/n_l$ converges, we construct 
 the sequence $(n_l)_{l\geqslant 1}$ by induction; once $n_i, i\leqslant l-1$ are defined, 
 we choose $n_l$ such that $n_l\geqslant l^2k_ln_{l-1}$ and \eqref{lac4} holds.
 
 Using Rokhlin's lemma, we can find for any integer $l\geqslant 1$ a measurable set $
 C_l\in\mathcal C$ such that the sets $T^{-i}C_l$, $i=0,\dots,n_l-1$ are 
 pairwise disjoint and $\mu\left(\bigcup_{i=0}^{n_l-1}T^{-i}C_l\right)>1/2$. 
 
 For a fixed $l$, we define 
 \begin{equation}
  k_{l,j}:=2^{I_l+J_l-j}, \quad 0\leqslant j\leqslant J_l,
 \end{equation}
 \begin{multline}
 k_{l,j}:=2^{I_l+J_l-j}, \quad 0\leqslant j\leqslant J_l \mbox{ and } \\
  f_l:=\frac 1{L_l}\sum_{j=0}^{J_l-1}\left(\frac{n_l}{k_{l,J_l-j}}\right)^{1/p}
  \chi\left(\bigcup_{i=k_{l,J_l-j}}^{k_{l,J_l-j-1}-1}T^{-i}C_l\right)+\\
  +\frac 1{L_l}\left(\frac{n_l}{k_{l,J_l}}\right)^{1/p}\chi\left(\bigcup_{i=0}^{k_{l,J_l}-1}
  T^{-i}C_l\right),
 \end{multline}
 \begin{equation}\label{eq:definition_f_m}
  f:= \sum_{l=1}^{+\infty}f_l, \quad m:=g\cdot f,
 \end{equation}
 where $g$ is the function obtained by Lemma~\ref{lem:positive_entropy}.
 
 \begin{propo}\label{decay_tail}
  We have the estimate $\lVert f_l\rVert_{p,\infty}\leqslant \kappa'_p L_l^{-1}$ for 
  some constant $\kappa'_p$ depending only on $p$. As a consequence,
  $\lim_{t\to \infty}t^p\mu\ens{\abs m>t}=0$.
 \end{propo}
 
 \begin{proof}
  Notice that 
  \begin{equation}\label{norm_f_l_estimation1}
   \norm{\frac 1{L_l}\left(\frac{n_l}{k_{l,J_l}}\right)^{1/p}\chi\left(\bigcup_{i=0}^{k_{l,J_l}-1}
  T^{-i}C_l\right)}_{p,\infty}^p=\frac 1{L_l^p}\frac{n_l}{k_{l,J_l}}
  k_{l,J_l}\cdot\mu(C_l)\leqslant \frac 1{L_l^p}.
  \end{equation}

 Next, using Lemma~\ref{estimate_norm_simple_function} with $N:=J_l-1$, 
 $a_j:=\frac 1{L_l}\left(\frac{n_l}{k_{l,J_l-j}}\right)^{1/p}$ and 
 $A_j:= \bigcup_{i=k_{l,J_l-j}}^{k_{l,J_l-j-1}-1}T^{-i}C_l$, we obtain 
 \begin{align}
  \norm{\frac 1{L_l}\sum_{j=0}^{J_l-1}\left(\frac{n_l}{k_{l,J_l-j}}\right)^{1/p}
  \chi\left(\bigcup_{i=k_{l,J_l-j}}^{k_{l,J_l-j-1}}T^{-i}C_l\right)}_{p,\infty}^p
  &\leqslant \max_{0\leqslant j\leqslant J_l-1}\left(\frac 1{L_l}
  \left(\frac{n_l}{k_{l,J_l-j}}\right)^{1/p}
  \right)^p  \sum_{i=0}^j\mu(A_j)\\
  &\leqslant \frac 1{L_l^p}\max_{0\leqslant j\leqslant J_l-1}\frac{n_l}{k_{l,J_l-j}}
  \sum_{i=0}^j\frac{k_{l,J_l-i}}{n_l}\\
  &=\frac 1{L_l^p}\max_{0\leqslant j\leqslant J_l-1}\sum_{i=0}^j\frac{2^{I_l+i}}{2^{I_l+j}}\\
  &\leqslant \frac 2{L_l^p}\label{norm_f_l_estimation2},
 \end{align}
 hence by \eqref{link_weak_lp_norms}, \eqref{norm_f_l_estimation1} and \eqref{norm_f_l_estimation2},
 \begin{align}
  \norm{f_l}_{p,\infty} &\leqslant
  N_p\left(\frac 1{L_l}\left(\frac{n_l}{k_{l,J_l}}\right)^{1/p}\chi\left(\bigcup_{i=0}^{k_{l,J_l}-1}
  T^{-i}C_l\right) \right)+\\
  &+ N_p\left(\frac 1{L_l}\sum_{j=0}^{J_l-1}\left(\frac{n_l}{k_{l,J_l-j}}\right)^{1/p}
  \chi\left(\bigcup_{i=k_{l,J_l-j}}^{k_{l,J_l-j-1}}T^{-i}C_l\right) \right)\\
  &\leqslant \kappa_p \norm{\frac 1{L_l}\left(\frac{n_l}{k_{l,J_l}}\right)^{1/p}
  \chi\left(\bigcup_{i=0}^{k_{l,J_l}-1}
  T^{-i}C_l\right)}_{p,\infty}+\\
  &+\kappa_p\norm{\frac 1{L_l}\sum_{j=0}^{J_l-1}\left(\frac{n_l}{k_{l,J_l-j}}\right)^{1/p}
  \chi\left(\bigcup_{i=k_{l,J_l-j}}^{k_{l,J_l-j-1}}T^{-i}C_l\right)}_{p,\infty}\\
 &\leqslant  \frac 1{L_l}\kappa_p\left(1+2^{1/p}\right)
 \end{align}

  We thus define $\kappa'_p:=\kappa_p\left(1+2^{1/p}\right)$.
 
 We fix $\varepsilon>0$; using \eqref{lac0}, we can find an integer $l_0$ such that 
 $\sum_{l>l_0}1/L_l<\varepsilon$. Since the function $\sum_{l=1}^{l_0}gf_l$ is bounded, 
 we have, 
 \begin{align}
  \limsup_{t\to \infty}t^p\mu\ens{\abs m>t}&\leqslant 
  \limsup_{t\to \infty}t^p\mu\ens{\abs{\sum_{l=1}^{l_0}gf_l}>\frac t2}+
  2^p\norm{\sum_{l>l_0}gf_l}_{p,\infty}^p\\
  &=2^p\norm{\sum_{l>l_0}gf_l}_{p,\infty}^p\\
  &\leqslant \left(2\sum_{l>l_0}N_p(f_l)\right)^p\\
  &\leqslant \kappa'_p\left(\sum_{l>l_0}\frac 1{L_l}\right)^p\\
  &\leqslant \kappa'_p\varepsilon^p,
 \end{align}
 where the second inequality comes from inequalities \eqref{link_weak_lp_norms}.
 Since $\varepsilon$ is arbitrary, the proof of Lemma~\ref{decay_tail} is complete. 
 \end{proof}

   We denote by $\mathcal M$ the $\sigma$-algebra generated by $\mathcal C$ 
  and the random variables $g\circ T^k$, $k\leqslant 0$. It satisfies 
 $\mathcal M\subset T^{-1}\mathcal M$. 
 \begin{propo}\label{martingale_difference}
  The sequence $(m\circ T^i)_{i\geqslant 0}$ is a (stationary) 
  martingale difference sequence with 
  respect to the filtration $(T^{-i}\mathcal M)_{i\geqslant 0}$. 
 \end{propo}
 
 \begin{proof}
  We have to show that $\mathbb E[m\mid T\mathcal M]=0$. Since the $\sigma$-algebra 
  $\mathcal C$ is $T$-invariant, we have $T\mathcal M=\sigma(\mathcal C\cup \sigma(g\circ T^k, k\leqslant 
  -1))$. This implies 
  \eq{
   \mathbb E[m\mid T\mathcal M]=\mathbb E[gf\mid T\mathcal M]=f\cdot\mathbb E[g\mid T\mathcal M].
  }
  Since $g$ is centered and independent of $T\mathcal M$, Proposition~\ref{martingale_difference} 
  is proved. 
 \end{proof}

 It remains to prove that the process $(n^{-1/2}S_n^{\mathrm{pl}}(m))_{n\geqslant 1}$ is not 
  tight in $\mathcal H^o_{1/2-1/p}[0,1]$. 
  
 \begin{propo}\label{below_bound_gf_l}
  Under conditions \eqref{lac3}, \eqref{lac2} and 
  \eqref{lac1}, there exists an integer $l_0$ such that for $l\geqslant l_0$ 
  \begin{equation}\label{eq:lower_bound_gf_l}
   P_l:=\mu\ens{\frac 1{n_l^{1/p}}\quad\max_{\mathclap{\substack{1\leqslant u\leqslant n_l-k_l\\ 
   1\leqslant v\leqslant k_l}}}\frac{\abs{S_{u+v}(gf_l)-S_u(gf_l)}}{v^{1/2-1/p}}\geqslant 1}
   \geqslant \frac 1{16}.
  \end{equation}
 \end{propo}
 \begin{proof}
  Let us fix an integer $l\geqslant 1$. Assume that $\omega\in T^{-s}C_l$, where 
  $k_l\leqslant s\leqslant n_l-1$. Since $T^u\omega$ belongs to $T^{-(s-u)}C_l$ we have 
  for $s-n_l\leqslant u\leqslant s$
  \begin{equation}\label{eq:iterates_of_f_l}
   (f_l\circ T^u)(\omega)=\begin{cases}
    \frac 1{L_l}\left(\frac{n_l}{k_{l,J_l}}\right)^{1/p},&\quad \mbox{if } 
    s-k_{l,J_l}<u\leqslant s;\\
    \frac 1{L_l}\left(\frac{n_l}{k_{l,j}}\right)^{1/p},&\quad \mbox{if } 
    s-k_{l,j-1}< u\leqslant s-k_{l,j}, 
    \mbox{ and }1\leqslant j\leqslant J_l;\\
    0,&\quad\mbox{if }s-n_l \leqslant u< s-k_l. 
   \end{cases}
  \end{equation}
As a consequence, 
 \begin{multline}\label{C_l}
  T^{-s}C_l\cap \ens{\frac 1{n_l^{1/p}}\max_{1\leqslant j\leqslant J_l}
  \frac{\abs{S_{s-k_{l,j-1}+1}(gf_l)-S_{s-k_{l,j}}(gf_l)}}
  {(k_{l,j-1}-1-k_{l,j})^{1/2-1/p}}\geqslant 1}\\
  =T^{-s}C_l\cap \ens{\max_{1\leqslant j\leqslant J_l}
  \frac{\abs{S_{s-k_{l,j-1}+1}(g)-S_{s-k_{l,j}}(g)}}
  {(k_{l,j}-1)^{1/2-1/p}k_{l,j-1}^{1/p}}\geqslant L_l}.
 \end{multline}
Since for $k_l+1\leqslant s\leqslant n_l-k_l$ and $1\leqslant j\leqslant J_l$, 
we have $1\leqslant s-k_{l,j}\leqslant n_l-k_l$ and $
1\leqslant k_{l,j-1}-1-k_{l,j} \leqslant k_l$, the inequality 
\begin{multline}\label{eq:maximum_on_T^{-s}C_l}
 \chi(T^{-s}C_l)\cdot \max_{1\leqslant j\leqslant J_l}
  \frac{\abs{S_{s-k_{l,j-1}+1}(gf_l)-S_{s-k_{l,j}}(gf_l)}}
  {(k_{l,j-1}-1-k_{l,j})^{1/2-1/p}}\\
  \leqslant \chi(T^{-s}(C_l))\quad \max_{\mathclap{\substack{1\leqslant u\leqslant n_l-k_l\\ 
   1\leqslant v\leqslant k_l}}}\frac{\abs{S_{u+v}(gf_l)-S_u(gf_l)}}{v^{1/2-1/p}}
\end{multline}
takes place and since the sets $(T^{-s}C_l)_{s=0}^{n_k-1}$ are pairwise disjoint, we obtain 
the lower bound 
\begin{equation}\label{eq:below_bound_Pl}
 P_l\geqslant 
 \sum_{s=1}^{n_l-2k_l}\mu\left(T^{-(s+k_l)}(C_l)\cap\ens{\max_{1\leqslant j\leqslant J_l}
  \frac{\abs{S_{s+k_l-k_{l,j-1}+1}(gf_l)-S_{s+k_l-k_{l,j}}(gf_l)}}
  {(k_{l,j-1}-1-k_{l,j})^{1/2-1/p}}\geqslant 1} \right).
\end{equation}

Using the fact that $T$ is measure-preserving, this becomes 
\begin{equation}\label{eq:below_bound_Pl_2}
 P_l\geqslant(n_l-2k_l)\cdot \mu\left(T^{-k_l}(C_l)\cap\ens{\max_{1\leqslant j\leqslant J_l}
  \frac{\abs{S_{k_l-k_{l,j-1}-1}(gf_l)-S_{k_l-k_{l,j}}(gf_l)}}
  {(k_{l,j-1}-1-k_{l,j})^{1/2-1/p}}\geqslant 1} \right),
  \end{equation}
and plugging \eqref{C_l} in the previous estimate, we get 
\begin{equation}\label{eq:below_bound_Pl_3}
 P_l\geqslant (n_l-2k_l)\mu\left(T^{-k_l}(C_l)\cap \ens{\max_{1\leqslant j\leqslant J_l}
  \frac{\abs{S_{k_l-k_{l,j-1}-1}(g)-S_{k_l-k_{l,j}}(g)}}
  {(k_{l,j}-1)^{1/2-1/p}k_{l,j-1}^{1/p}}\geqslant L_l}\right).
\end{equation}
The sets $\ens{\max_{1\leqslant j\leqslant J_l}
  \frac{\abs{S_{k_l-k_{l,j-1}+1}(g)-S_{k_l-k_{l,j}}(g)}}
  {(k_{l,j}-1)^{1/2-1/p}k_{l,j-1}^{1/p}}\geqslant L_l}$ and $T^{-k_l}C_l$ belong 
  to the independent sub-$\sigma$-algebras $\mathcal B$ and $\mathcal C$ respectively, hence 
  using the fact that the sequences $(g\circ T^i)_{i\geqslant 0}$ and 
  $(g\circ T^{-i})_{i\geqslant 0}$ are identically distributed, we obtain
\begin{equation}\label{eq:below_bound_Pl_4}
 P_l\geqslant (n_l-2k_l)\mu\left(C_l\right)\mu\ens{\max_{1\leqslant j\leqslant J_l}
  \frac{\abs{S_{k_{l,j-1}-1}(g)-S_{k_{l,j}}(g)}}
  {(k_{l,j}-1)^{1/2-1/p}k_{l,j-1}^{1/p}}\geqslant L_l}.
\end{equation}

By construction, we have $n_l\cdot\mu(C_l)= \mu\left(\bigcup_{i=0}^{n_l-1}T^{-i}C_l\right)>1/2$, 
hence 
\begin{equation}\label{estimate_P_l}
 P_l\geqslant \frac 12\left(1-2\frac{k_l}{n_l}\right)
 \mu\ens{\max_{1\leqslant j\leqslant J_l}
  \frac{\abs{S_{k_{l,j-1}-1}(g)-S_{k_{l,j}}(g)}}
  {(k_{l,j}-1)^{1/2-1/p}k_{l,j-1}^{1/p}}\geqslant L_l}.
\end{equation}
It remains to find a lower bound for 
\begin{equation}\label{eq:definition_P'_l}
 P'_l:=\mu\ens{\max_{1\leqslant j\leqslant J_l}
  \frac{\abs{S_{k_{l,j-1}-1}(g)-S_{k_{l,j}}(g)}}
  {(k_{l,j}-1)^{1/2-1/p}k_{l,j-1}^{1/p}}\geqslant L_l}.
\end{equation}

Let us define the set
\begin{equation}\label{eq:definition_E_j}
 E_j:=\ens{
  \frac{\abs{S_{k_{l,j-1}-1}(g)-S_{k_{l,j}}(g)}}
  {(k_{l,j}-1)^{1/2-1/p}k_{l,j-1}^{1/p}}\geqslant L_l}
\end{equation}

Since the sequence $(g\circ T^i)_{i\in \Z}$ is independent, 
the family $(E_j)_{1\leqslant j\leqslant J_l}$ is independent, hence 
\begin{equation}
 P'_l\geqslant 1-\prod_{j=1}^{J_l}(1-\mu(E_j)) \label{estimate_P'_l}.
\end{equation}
We define the quantity
\begin{equation}\label{eq:definition_of_c_j}
 c_j:=\mu\ens{\abs{\mathcal N}\geqslant \frac{L_l}{\norm g_2}
 \left(\frac{k_{l,j-1}}{k_{l,j}-1}\right)^{1/p}}
\end{equation}

(we recall that $\mathcal N$ denotes a standard normally distributed random variable). 
By the Berry-Esseen theorem, we have for each $j\in \ens{1,\dots, J_l}$, 
\begin{equation}
 \abs{\mu(E_j)-c_j}\leqslant \frac 1{\norm{g}_2^3}\frac 1{(k_{l,j-1}-1)^{1/2}}\leqslant 
 \frac{\sqrt 2}{\norm{g}_2^3}2^{-I_l/2}\label{Berry_Essen}.
\end{equation}

Plugging the estimate \eqref{Berry_Essen} into \eqref{estimate_P'_l} and 
noticing that for an integer $N$ and $(a_n)_{n=1}^N$, $(b_n)_{n=1}^N$ two families of 
numbers in the unit interval,
\eq{
\abs{\prod_{n=1}^Na_n-\prod_{n=1}^Nb_n}\leqslant 
\sum_{n=1}^N\abs{a_n-b_n}, 
}
we obtain 
\begin{align}
 P'_l&\geqslant 1-\prod_{j=1}^{J_l}(1-\mu(E_j))+\prod_{j=1}^{J_l}(1-c_j)-\prod_{j=1}^{J_l}(1-c_j)\\
 &\geqslant 1-\prod_{j=1}^{J_l}(1-c_j)-\sum_{j=1}^{J_l}\abs{\mu(E_j)-c_j}\\
 &\geqslant 1-\prod_{j=1}^{J_l}(1-c_j)-J_l\frac{\sqrt 2}{\norm{g}_2^3}2^{-I_l/2}.
\end{align}
Notice that 
\eq{
 1-\prod_{j=1}^{J_l}(1-c_j)\geqslant 1-\max_{1\leqslant j\leqslant J_l}\left(
1- c_j\right)^{J_l}
}
and since $(I_l)_{\geqslant }$ is increasing and $I_1\geqslant 1$, we have 
\begin{equation}\label{eq:simplification_of_c_j}
 \frac{k_{l,j-1}}{k_{l,j}-1}=\frac 2{1-k_{l,j}^{-1}}\leqslant \frac 2{1-2^{-I_l}}
 \leqslant 4
\end{equation}
it follows by \eqref{eq:definition_of_c_j} that 
$c_j\geqslant  \mu\ens{\abs{\mathcal N}\geqslant 
4^{1/p}\frac{L_l}{\norm g_2}}$ for $1\leqslant j\leqslant J_l$. We thus have 
\begin{equation}\label{eq:lower_bound_P'_l}
 P'_l\geqslant 1-\left(1-\mu\ens{\abs{\mathcal N}\geqslant 
4^{1/p}\frac{L_l}{\norm g_2}}\right)^{J_l}-J_l\frac{\sqrt 2}{\norm{g}_2^3}2^{-I_l/2}.
\end{equation}

Using the elementary inequality 
\eq{
 1-(1-t)^n\geqslant nt-\frac{n(n-1)}2t^2
}
valid for a positive integer $n$ and $t\in [0,1]$, we obtain 
\begin{equation}\label{estimate_P'_l_bis}
 P'_l\geqslant J_l\mu\ens{\abs{\mathcal N}\geqslant 
4^{1/p}\frac{L_l}{\norm g_2}} -\frac{J_l^2}2\left(\mu\ens{\abs{\mathcal N}\geqslant 
4^{1/p}\frac{L_l}{\norm g_2}}\right)^2-J_l\frac{\sqrt 2}{\norm{g}_2^3}2^{-I_l/2}.
\end{equation}
By conditions \eqref{lac2} and \eqref{lac3}, there exists an integer $l'_0$ such that if 
$l\geqslant l'_0$, then 
\begin{equation}\label{estimate_P'_l_ter}
 \mu\ens{\max_{1\leqslant j\leqslant J_l}
  \frac{\abs{S_{k_{l,j-1}-1}(g)-S_{k_{l,j}}(g)}}
  {(k_{l,j}-1)^{1/2-1/p}k_{l,j-1}^{1/p}}\geqslant L_l}\geqslant \frac 14.
\end{equation}
Combining \eqref{estimate_P_l} with \eqref{estimate_P'_l_ter}, we obtain for $l\geqslant l'_0$
\begin{equation}\label{eq:lower_bound_Pl_6}
 P_l\geqslant \frac 18\left(1-2\frac{k_l}{n_l}\right).
\end{equation}
By condition \eqref{lac1}, we thus get that $P_l\geqslant 1/16$ for $l\geqslant l_0$, where 
$l_0\geqslant l'_0$ and $k_l/n_l\leqslant 1/4$ if $l\geqslant l_0$. 

This concludes the proof of Proposition~\ref{below_bound_gf_l}. 
 \end{proof}

\begin{propo}\label{Propo_bound_norm_Snpl}
 Under conditions \eqref{lac0}, \eqref{lac3}, \eqref{lac2}, \eqref{lac1} and \eqref{lac4}, 
 we have for $l$ large enough
 \begin{equation}\label{bound_norm_Snpl}
  \mu\ens{\frac 1{n_l^{1/p}}\quad\max_{\mathclap{\substack{1\leqslant u\leqslant n_l-k_l\\ 
   1\leqslant v\leqslant k_l}}}\frac{\abs{S_{u+v}(m)-S_u(m)}}{v^{1/2-1/p}}\geqslant \frac 12}
   \geqslant \frac 1{32}.
 \end{equation}
\end{propo}

Since the H\"older modulus of continuity of a piecewise linear function is reached at 
vertices, we derive the following corollary.

\begin{cor}\label{corollary_non_tightness}
If $l\geqslant l_0$, then 
\begin{equation}
 \mu\ens{\omega_{1/2-1/p}\left(\frac 1{\sqrt{n_l}}S_{n_l}^{\mathrm{pl}}(m), \frac{k_l}{n_l}\right)
 \geqslant \frac 12}\geqslant \frac 1{32}.
\end{equation}
Therefore, for each positive $\delta$, we have 
\begin{equation}
 \limsup_{n\to \infty}\mu\ens{\omega_{1/2-1/p}\left(
 \frac 1{\sqrt{n}}S_{n}^{\mathrm{pl}}(m),\delta\right)
 \geqslant \frac 12}\geqslant \frac 1{32},
\end{equation}
and the process $(n^{-1/2}S_n^{\mathrm{pl}}(m))_{n\geqslant 1}$ is not 
  tight in $\mathcal H^o_{1/2-1/p}[0,1]$.
\end{cor}

\begin{proof}[Proof of Proposition~\ref{Propo_bound_norm_Snpl}]

Let $l_0$ be the integer given by Proposition~\ref{below_bound_gf_l} and let $l\geqslant l_0$.
We define $m'_l:=\sum_{i=1}^{l-1}gf_i$ and $m''_l:=\sum_{i=l+1}^{+\infty}gf_i$.

We define for $i\geqslant 1$,
 \begin{equation}\label{eq:definition_Mli}
 M_{l,i}:= \frac 1{n_l^{1/p}}\quad\max_{\mathclap{\substack{1\leqslant u\leqslant n_l-k_l\\ 
  1\leqslant  v\leqslant k_l}}}\frac{\abs{S_{u+v}(gf_i)-S_u(gf_i)}}{v^{1/2-1/p}}.
 \end{equation}
 Let $i$ be an integer such that $i<l$.
Notice that for $1\leqslant u\leqslant n_l-k_l$ and $ v\leqslant k_l$, we have 
\eq{
 \abs{S_{u+v}(gf_i)-S_u(gf_i)}=U^u(\abs{S_{v}(gf_i)}),
}
where $U(h)(\omega)=h\left(T\left(\omega\right)\right)$ and since 
\eq{
 \abs{S_{v}(gf_i)}\leqslant v\norm{gf_i}_{\infty}\leqslant \frac{k_l}{L_i}
 \left(\frac{n_i}{2^{I_i}}\right)^{1/p},
}
the estimate 
\begin{equation}\label{estimate_M_l,i}
 M_{l,i}\leqslant \frac{k_l}{L_in_l^{1/p}}\left(\frac{n_i}{2^{I_i}}\right)^{1/p}
\end{equation}
holds. Since 
\begin{equation}\label{eq:Mli_i_leq_l-1}
 \frac 1{n_l^{1/p}}\quad\max_{\mathclap{\substack{1\leqslant u\leqslant n_l-k_l\\ 
 1\leqslant  v\leqslant k_l}}}\frac{\abs{S_{u+v}(m'_l)-S_u(m'_l)}}{v^{1/2-1/p}}\leqslant 
   \sum_{i=1}^{l-1}M_{l,i},
\end{equation}
we have by \eqref{estimate_M_l,i}, 
\begin{equation}\label{eq:Mli_i_leq_l-1_bis}
 \frac 1{n_l^{1/p}}\quad\max_{\mathclap{\substack{1\leqslant u\leqslant n_l-k_l\\ 
  1\leqslant v\leqslant k_l}}}\frac{\abs{S_{u+v}(m'_l)-S_u(m'_l)}}{v^{1/2-1/p}}\leqslant
  \sum_{i=1}^{l-1} \frac{k_l}{L_in_l^{1/p}}\left(\frac{n_i}{2^{I_i}}\right)^{1/p}.
\end{equation}
By \eqref{lac4}, the following bound takes place: 
\begin{equation}\label{borne_M_li,i<l}
 \frac 1{n_l^{1/p}}\quad\max_{\mathclap{\substack{1\leqslant u\leqslant n_l-k_l\\ 
  1\leqslant v\leqslant k_l}}}\frac{\abs{S_{u+v}(m'_l)-S_u(m'_l)}}{v^{1/2-1/p}}\leqslant\frac 12.
\end{equation}

The following set inclusions hold
\begin{align}\label{eq:bound_m''_l}
 \ens{\frac 1{n_l^{1/p}}\quad\max_{\mathclap{\substack{1\leqslant u\leqslant n_l-k_l\\ 
 1\leqslant  v\leqslant k_l}}}\frac{\abs{S_{u+v}(m''_l)-S_u(m''_l)}}{v^{1/2-1/p}}\neq 0}
   &\subset\bigcup_{i>l}
   \ens{M_{l,i}\neq 0}\\
   & \subset\bigcup_{i>l}\bigcup_{u=1}^{n_l}\ens{U^u(gf_i)\neq 0}.
\end{align}
We thus have 
\begin{align}\label{eq:bound_m''_l_bis}
 \mu\ens{\frac 1{n_l^{1/p}}\quad\max_{\mathclap{\substack{1\leqslant u\leqslant n_l-k_l\\ 
 1\leqslant  v\leqslant k_l}}}\frac{\abs{S_{u+v}(m''_l)-S_u(m''_l)}}{v^{1/2-1/p}}\neq 0}
   &\leqslant \sum_{i>l}n_l\cdot\mu\ens{gf_i\neq 0}\\
   &\leqslant n_l\sum_{i>l}\mu\ens{f_i\neq 0}\\
   &=n_l\sum_{i>l}(k_i+1)\mu(C_i)\\
   &\leqslant 2n_l\sum_{i>l}\frac{k_i}{n_i}.
\end{align}
and by \eqref{lac1}, it follows that 
\begin{equation}\label{borne_M_li,i>l}
 \mu\ens{\frac 1{n_l^{1/p}}\quad\max_{\mathclap{\substack{1\leqslant u\leqslant n_l-k_l\\ 
 1\leqslant  v\leqslant k_l}}}\frac{\abs{S_{u+v}(m''_l)-S_u(m''_l)}}{v^{1/2-1/p}}\neq 0}
   \leqslant \frac 1{32}
\end{equation}
Accounting \eqref{borne_M_li,i<l}, we thus have 
\begin{multline}\label{eq:final_bound_m}
 \mu\ens{\frac 1{n_l^{1/p}}\quad\max_{\mathclap{\substack{1\leqslant u\leqslant n_l-k_l\\ 
  1\leqslant v\leqslant k_l}}}\frac{\abs{S_{u+v}(m)-S_u(m)}}{v^{1/2-1/p}}\geqslant \frac 12}\\
   \geqslant \mu\ens{\frac 1{n_l^{1/p}}\quad\max_{\mathclap{\substack{1\leqslant u\leqslant n_l-k_l\\ 
  1\leqslant v\leqslant k_l}}}\frac{\abs{S_{u+v}(gf_l+m''_l)-S_u(gf_l+m''_l)}}{v^{1/2-1/p}}\geqslant 1}\\
   \geqslant\mu\ens{\frac 1{n_l^{1/p}}\quad\max_{\mathclap{\substack{1\leqslant u\leqslant n_l-k_l\\ 
  1\leqslant v\leqslant k_l}}}\frac{\abs{S_{u+v}(gf_l)-S_u(gf_l)}}{v^{1/2-1/p}}\geqslant 1} \\
   -\mu\ens{\frac 1{n_l^{1/p}}\quad\max_{\mathclap{\substack{1\leqslant u\leqslant n_l-k_l\\ 
  1\leqslant v\leqslant k_l}}}\frac{\abs{S_{u+v}(m''_l)-S_u(m''_l)}}{v^{1/2-1/p}}\neq 0},
\end{multline}
hence combining Proposition~\ref{below_bound_gf_l} with \eqref{borne_M_li,i>l}, we obtain the 
conclusion of Proposition~\ref{Propo_bound_norm_Snpl}.
\end{proof}
 
Theorem~\ref{counter_example} follows from Corollary~\ref{corollary_non_tightness} and 
Propositions~\ref{decay_tail} and \ref{martingale_difference}.

 \subsection{Proof of Theorem~\ref{sufficient_cond_mart} and Proposition~\ref{generalized_Doob}}
 
 \begin{proof}[Proof of Proposition~\ref{generalized_Doob}]

 Let us fix a positive $t$.
 Recall the equivalence between $\norm{x}_{\alpha}$ and 
 $\norm{x}_{\alpha}^{\mathrm{seq}}$
 and Notation~\ref{first_notation}. By Remark~\ref{rmq:simplification_sequential_norm}, 
 we have to show that for some 
 constant $C$ depending only on $p$ and 
 each integer $n\geqslant 1$, 
 \begin{multline}\label{goal_inequality}
 P(n,t):= t^p\mu\ens{\sup_{j\geqslant 1} 2^{\alpha j}n^{-1/2}
  \max_{0\leqslant k<2^j}\abs{S_n^{\mathrm{pl}}(m,r_{k+1,j})-S_n^{\mathrm{pl}}(m,r_{k,j})}
  >t}\leqslant \\
  \leqslant C\left(\norm m_{p,\infty}^p+\mathbb E\left(\mathbb E[m^2\mid T\mathcal M]
  \right)^{p/2}\right)
 \end{multline}

 In the proof, we shall denote by $C_p$ a constant depending only on $p$ which may change from 
 line to line. 
 
 We define 
 \eq{
 P_1(n,t):= \mu\ens{\sup_{1\leqslant j\leqslant \log n} 2^{\alpha j}n^{-1/2}
  \max_{0\leqslant k<2^j}\abs{S_n^{\mathrm{pl}}(m,r_{k+1,j})-S_n^{\mathrm{pl}}(m,r_{k,j})}
  >t}, \mbox{ and }
 }
  \eq{P_2(n,t):= \mu\ens{\sup_{j> \log n} 2^{\alpha j}n^{-1/2}
  \max_{0\leqslant k<2^j}\abs{S_n^{\mathrm{pl}}(m,r_{k+1,j})-S_n^{\mathrm{pl}}(m,r_{k,j})}
  >t},
  }
  hence 
  \begin{equation}
   P(n,t)\leqslant t^pP_1(n,t/2)+t^pP_2(n,t/2).
  \end{equation}

  We estimate $P_2(n,t)$. For $j>\log n$, we have the inequality 
  \begin{equation}
    r_{k+1,j}-r_{k,j}=(k+1)2^{-j}-k2^{-j}=2^{-j}<1/n,
  \end{equation}
  hence if $r_{k,j}$ belongs to the interval $[l/n,(l+1)/n)$ for some $l\in 
  \ens{0,\dots,n-1}$, then 
  \begin{itemize}
   \item either $r_{k+1,j}\in [l/n,(l+1)/n)$, and in this case, 
   \eq{
    \abs{S_n^{\mathrm{pl}}(m,r_{k+1,j})-S_n^{\mathrm{pl}}(m,r_{k,j})}
    =\abs{m\circ T^{l+1}}2^{-j}n\leqslant 2^{-j}n\max_{1\leqslant l\leqslant n}
    \abs{U^l(m)};
   }
   \item or $r_{k+1,j}$ belongs to the interval $[(l+1)/n,(l+2)/n)$. The estimates 
   \begin{multline}
    \abs{S_n^{\mathrm{pl}}(m,r_{k+1,j})-S_n^{\mathrm{pl}}(m,r_{k,j})}\leqslant 
    \abs{S_n^{\mathrm{pl}}(m,r_{k+1,j})-S_n^{\mathrm{pl}}(m,(l+1)/n)}+\\
    +\abs{S_n^{\mathrm{pl}}(m,(l+1)/n)-S_n^{\mathrm{pl}}(m,r_{k,j})}
    \leqslant 2^{1-j}n\max_{1\leqslant l\leqslant n}
    \abs{U^l(m)}
   \end{multline}
  hold.
  \end{itemize}
 Considering these two cases, we obtain 
 \begin{align}
  P_2(n,t)&\leqslant \mu\ens{\sup_{j>\log n}2^{\alpha j}n2^{1-j}
  n^{-1/2}\max_{1\leqslant l\leqslant n}\abs{U^l(m)}>t}\\
  &\leqslant\mu\ens{2n^{\alpha-1/2}\max_{1\leqslant l\leqslant n}
    \abs{U^l(m)}>t}\\
  &\leqslant n\mu\ens{2n^{-1/p}\abs m>t}\\
  &\leqslant \frac{2^p}{t^p}\sup_{x> 0}x^p\mu\ens{\abs m>x}.
 \end{align}
 Therefore, establishing inequality \eqref{goal_inequality} reduces to find a 
 constant $C$ depending only on $p$ such that 
 \begin{equation}\label{goal_inequality2}
  \sup_n\sup_tt^pP_1(n,t)\leqslant C\left(\norm m_{p,\infty}^p+\mathbb E\left(
  \mathbb E[m^2\mid T\mathcal M]\right)^{p/2}\right)
 \end{equation}
 We define $u_{k,j}:=[nr_{k,j}]$ for $k<2^j$ and $j\geqslant 1$ (see Notation~\ref{first_notation}). 
 
 Notice that the inequalities 
 \eq{
              \abs{S_{u_{k,j}}(m)-
 S_n^{\mathrm{pl}}(m,r_{k,j})}\leqslant \abs{U^{u_{k,j}+1}(m)}\quad\mbox{and}
 }
 \eq{
  \abs{S_n^{\mathrm{pl}}(m,r_{k+1,j})-
 S_{u_{k+1,j}}(m)}\leqslant \abs{U^{u_{k+1,j}+1}(m)}
 }
 take place because if $j \leqslant \log n$, then 
 \eq{
  u_{k,j}\leqslant nr_{k,j}\leqslant u_{k,j}+1\leqslant u_{k+1,j}\leqslant 
  nr_{k+1,j}\leqslant u_{k+1,j}+1.
 }

 Therefore,  
 $P_1(n,t)\leqslant P_{1,1}(n,t)+P_{1,2}(n,t)$, where 
 \begin{align}
  P_{1,1}(n,t)&:=\mu\ens{\max_{1\leqslant j\leqslant \log n}
  2^{\alpha j}n^{-1/2}\max_{0\leqslant k<2^j}\abs{S_{u_{k+1,j}}(m)
  -S_{u_{k,j}}(m)}>t/2},\\
  P_{1,2}(n,t)&:=\mu\ens{\max_{1\leqslant j\leqslant \log n}
  2^{\alpha j}n^{-1/2}\max_{1\leqslant l\leqslant n}\abs{U^l(m)}>t/4}.
 \end{align}
 Notice that 
 \begin{align}
  P_{1,2}(n,t)&\leqslant \mu\ens{n^{\alpha-1/2}\max_{1\leqslant l\leqslant n}\abs{U^l(m)}>
  t/4}\\
  &\leqslant n\mu\ens{\abs m>n^{1/p}t/4}\\
  &\leqslant 4^pt^{-p}\sup_{x\geqslant 0}x^p\mu\ens{\abs m>x},
 \end{align}
 hence \eqref{goal_inequality2} will follow from the existence of a constant 
 $C$ depending only on $p$ such that 
 \begin{equation}
  \sup_n\sup_tt^pP_{1,1}(n,t)\leqslant C\left(\norm m_{p,\infty}^p+
  \mathbb E\left(\mathbb E[m^2\mid T\mathcal M]\right)^{p/2}\right).
 \end{equation}
 We estimate $P_{1,1}(n,t)$ in the following way:
 \begin{equation}\label{first_bound_P_{1,1}}
  P_{1,1}(n,t)\leqslant \sum_{j=1}^{\log n}
  2^j\max_{0\leqslant k<2^j}\mu\ens{\abs{S_{u_{k+1,j}}(m)
  -S_{u_{k,j}}(m)}>tn^{1/2}2^{-1-\alpha j}}
 \end{equation}
 We define for $1\leqslant j\leqslant \log n$ and $0\leqslant k<2^j$ the quantity 
 \begin{equation}
  \label{second_bound_P_{1,1}}
  P(n,j,k,t):=\mu\ens{\abs{S_{u_{k+1,j}}(m)
  -S_{u_{k,j}}(m)}>tn^{1/2}2^{-1-\alpha j}}.
 \end{equation}

 If $(f\circ T^j)_{j\geqslant 0}$ is a strictly stationary sequence, we define  
 \eq{
  Q_{f,n}(u):=\mu\ens{\max_{1\leqslant j\leqslant n}\abs{f\circ T^j}>u}+
  \mu\ens{\left(\sum_{i=1}^nU^i\mathbb E[f^2\mid T\mathcal M]\right)^{1/2}>u}.
 }
 The following inequality is Theorem~1 of \cite{MR2021875}. It allows us to express the 
 tail function of a martingale by that of the increments and the quadratic variance.
 \begin{theo}
  Let $m$ be an $\mathcal M$-measurable function such that $\mathbb E[m\mid T\mathcal M]=0$. 
  Then for each positive $y$ and each integer $n$, 
  \begin{equation}\label{nag}
   \mu\ens{\abs{S_n(m)}>y}\leqslant c(q,\eta)\int_0^1Q_{m,n}(\varepsilon_qu\cdot y)u^{q-1}\mathrm du,
  \end{equation}
  where $q>0$, $\eta>0$, $\varepsilon_q:=\eta/q$ and $c(q,\eta)
 :=q\exp(3\eta e^{\eta+1}-\eta-1)/\eta$. 
 \end{theo}
 We shall use \eqref{nag} with $q:=p+1$, $\eta=1$ and $y:=n^{1/2}2^{-1-\alpha j}t$
 in order to estimate $P(n,j,k,t)$:
 \begin{multline}\label{bound_P(n,j,k,t)}
  P(n,j,k,t)\leqslant C_p\int_0^1\mu\ens{\max_{1\leqslant i\leqslant u_{k+1,j}-u_{k,j}}\abs{U^i(m)}
 >n^{1/2}2^{-1-\alpha j}t u\varepsilon_{p+1}}u^p\mathrm du\\
 +C_p\int_0^1\mu\ens{\left(\sum_{i=u_{k,j}+1}^{u_{k+1,j}}U^i(\mathbb E[m^2\mid T\mathcal M]
 )\right)^{1/2}> n^{1/2}2^{-1-\alpha j}ut\varepsilon_{p+1}}u^p\mathrm du.
 \end{multline}
 Exploiting the inequality $u_{k+1,j}-u_{k,j}\leqslant 2n2^{-j}$, we get from the previous bound 
 \begin{multline}
 P(n,j,k,t)\leqslant C_p\int_0^1\mu\ens{\max_{1\leqslant i\leqslant 2n2^{-j}}\abs{U^i(m)}
 >n^{1/2}2^{-1-\alpha j}t u\varepsilon_{p+1}}u^p\mathrm du\\
 +C_p\int_0^1\mu\ens{\left(\sum_{i=1}^{2n2^{-j}}U^i(\mathbb E[m^2\mid T\mathcal M]
 )\right)^{1/2}> n^{1/2}2^{-1-\alpha j}tu\varepsilon_{p+1}}u^p\mathrm du.
\end{multline}
We define for $j\leqslant \log n$, $t\geqslant 0$ and $u\in (0,1)$, 
\eq{
 P'(n,j,t,u):=\mu\ens{\max_{1\leqslant i\leqslant 2n2^{-j}}\abs{U^i(m)}
 >n^{1/2}2^{-1-\alpha j}t u\varepsilon_{p+1}},\quad\mbox{ and }
}
\eq{
 P''(n,j,t,u):=\mu\ens{\left(\sum_{i=1}^{2n2^{-j}}U^i(\mathbb E[m^2\mid T\mathcal M]
 )\right)^{1/2}> n^{1/2}2^{-1-\alpha j}tu\varepsilon_{p+1}}.
}
Using the fact that the random variables $U^i(m), 1\leqslant i\leqslant 2n2^{-j}$ are 
identically distributed, we derive the bound 
\eq{
P'(n,j,t,u)\leqslant  2n2^{-j}\mu\ens{\abs m>n^{1/2}2^{-1-\alpha j}t u\varepsilon_{p+1}},
}
hence 
\begin{multline}
P'(n,j,t,u)\leqslant 
 2n2^{-j}(n^{1/2}2^{-1-\alpha j}t u\varepsilon_{p+1})^{-p}
 \norm{m}_{p,\infty}^p\\ 
 =2^{p+1}\varepsilon_{p+1}^{-p} n^{1-p/2}2^{j(-1+p\alpha)}t^{-p}u^{-p}\norm{m}_{p,\infty}^p.
\end{multline}
Since $\alpha$ and $p$ are linked by the relationship $1/2-1/p=\alpha$, we have 
$p\alpha=p/2-1$ hence 
\begin{equation}\label{bound_P'(n,j,t,u)}
 \int_0^1P'(n,j,t,u)u^p\mathrm{d}u\leqslant C_p
 t^{-p}n^{1-p/2}2^{j(p/2-2)}\norm{m}_{p,\infty}^p.
\end{equation}
Notice the following set equalities: 
\begin{multline}
 \ens{\left(\sum_{i=1}^{2n2^{-j}}U^i(\mathbb E[m^2\mid T\mathcal M]
 )\right)^{1/2}>\varepsilon_{p+1} un^{1/2}2^{-1-\alpha j}t}\\
 =\ens{\frac 1{2n2^{-j}}\sum_{i=1}^{2n2^{-j}}U^i(\mathbb E[m^2\mid T\mathcal M]
 )>2^{-3}\varepsilon_{p+1}^2 u^22^{2j/p}t^2}
\end{multline}
and that $n2^{-j}\geqslant 1$ (because $j \leqslant \log n$), hence 
\begin{multline}
 \ens{\left(\sum_{i=1}^{2n2^{-j}}U^i(\mathbb E[m^2\mid T\mathcal M]
 )\right)^{1/2}>\varepsilon_{p+1} un^{1/2}2^{-1-\alpha j}t}\subseteq
\\ \subseteq \bigcup_{N\geqslant 2}\ens{\frac 1{N}\sum_{i=1}^NU^i(\mathbb E[m^2\mid T\mathcal M]
 )>2^{-3}\varepsilon_{p+1}^2 u^22^{2j/p}t^2},
\end{multline}
from which it follows 
\begin{multline}\label{eq_variance_quadra}
 \mu\ens{\left(\sum_{i=1}^{2n2^{-j}}U^i(\mathbb E[m^2\mid T\mathcal M]
 )\right)^{1/2}>\varepsilon_{p+1} un^{1/2}2^{-1-\alpha j}t}\leqslant \\
 \leqslant \mu\ens{\sup_{N\geqslant 2}\frac 1{N}\sum_{i=1}^NU^i(\mathbb E[m^2\mid T\mathcal M]
 )>2^{-3}\varepsilon_{p+1}^2 u^22^{2j/p}t^2}.
\end{multline}
Combining \eqref{bound_P'(n,j,t,u)} and \eqref{eq_variance_quadra}, we obtain 
\begin{multline}
 \max_{0\leqslant k< 2^j}P(n,j,k,t)\leqslant 
  C_pt^{-p}n^{1-p/2}2^{j(p/2-2)}\norm{m}_{p,\infty}^p\\
  +C_p\int_0^1\mu\ens{\sup_{N\geqslant 2}\frac 1{N}\sum_{i=1}^NU^i(\mathbb E[m^2\mid T\mathcal M]
 )>2^{-3}\varepsilon_{p+1}^2 u^22^{2j/p}t^2}u^p\mathrm du,
\end{multline}
hence by \eqref{first_bound_P_{1,1}} and \eqref{second_bound_P_{1,1}},
\begin{multline}
P_{1,1}(n,t)\leqslant C_pt^{-p}\norm{m}_{p,\infty}^p\sum_{j=1}^{\log n}2^j2^{j(p/2-2)}n^{1-p/2}+\\
 +C_p\int_0^1\sum_{j=1}^{\log n}2^j\mu\ens{\sup_{N\geqslant 2}\frac 1{N}\sum_{i=1}^NU^i(\mathbb E[m^2\mid T\mathcal M]
 )>2^{-3}\varepsilon_{p+1}^2 u^22^{2j/p}t^2}u^p\mathrm du.
\end{multline}
From the elementary bounds 
\begin{align}
 \sum_{j=1}^{\log n}2^{j(p/2-1)}n^{1-p/2}&\leqslant 
(1-2^{1-p/2})^{-1} \\
\label{eq:norm_Lp_series}\sum_{j\geqslant 1}2^j\mu\ens{\abs g>2^{2j/p}}&\leqslant 
2\mathbb E\abs{g}^{p/2}, \quad 
\mbox{ for any non-negative function }g,
\end{align}
with 
\begin{equation}
 g:=2^{3}\varepsilon_{p+1}^{-2}u^{-2}
 \sup_{N\geqslant 2}\frac 1{N}\sum_{i=1}^NU^i(\mathbb E[m^2\mid T\mathcal M] ), u\in (0,1)
\end{equation}

 we obtain 
\begin{equation}\label{bound_P_{1,1}}
 P_{1,1}(n,t)\leqslant C_p t^{-p}\norm{m}_{p,\infty}^p+
 C_pt^{-p}\norm{\sup_{N\geqslant 2}\frac 1N\sum_{i=1}^N U^i(\mathbb E[m^2\mid T\mathcal M])
 }_{p/2}^{p/2}.
\end{equation}
As the Koopman operator $U$ is an $\mathbb L^1$-$\mathbb L^\infty$ contraction, 
Theorem~1 of \cite{MR0131517} gives the existence of a constant $A_p$ such that for each $h\in
\mathbb L^{p/2}$, 
\begin{equation}\label{bound_birkhoff_sums}
 \norm{\sup_{N\geqslant 1}\frac 1N\sum_{j=1}^NU^j(h)}_{p/2}\leqslant A_p\norm h_{p/2}.
\end{equation}

Applying \eqref{bound_birkhoff_sums} with $h:=\mathbb E[m^2\mid T\mathcal M]$, 
 we get by \eqref{bound_P_{1,1}} 
\begin{equation}
 P_{1,1}(n,t)\leqslant C_p t^{-p}\norm{m}_{p,\infty}^p+
 C_pt^{-p}\mathbb E\left(\mathbb E[m^2\mid T\mathcal M]\right)^{p/2},
\end{equation}
which establishes \eqref{goal_inequality2}. This concludes the proof 
of Proposition~\ref{generalized_Doob}.
\end{proof}

\begin{proof}[Proof of Theorem~\ref{sufficient_cond_mart}]
 
The convergence of finite dimensional distributions can be proved using 
Theorem~ of \cite{MR1324786}. Its proof works for filtrations of the 
form $(T^{-i}\mathcal M)_{i\geqslant 0}$ where $T\mathcal M \subset 
\mathcal M$ and also in the non-ergodic setting by considering the ergodic 
components.

We deduce tightness in Theorem~\ref{sufficient_cond_mart} from
Proposition~\ref{generalized_Doob} 
by a truncation argument. For a fixed $R$, we define 
\begin{equation}\label{eq:definition_m_R}
 m_R:=m\chi\ens{\abs m\leqslant R}-\mathbb E[m\chi\ens{\abs m\leqslant R}\mid T\mathcal M] 
 \quad \mbox{and}
\end{equation}
\eq{
 m'_R:=m\chi\ens{\abs m> R}-\mathbb E[m\chi\ens{\abs m> R}\mid T\mathcal M].
}
In this way, the sequences $(m_R\circ T^i)_{i\geqslant 0}$ and 
$(m'_R\circ T^i)_{i\geqslant 0}$ are martingale differences sequences and 
$m=m_R+m'_R$. 

Since $\abs{m_R}\leqslant 2R$ and $(m_R\circ T^i)_{i\geqslant 0}$ is a martingale 
difference sequence, the sequence $(n^{-1/2}S_n^{\mathrm{pl}}(m_R))_{n\geqslant 1}$ 
is tight in $\mathcal H^o_{1/2-1/p}[0,1]$. Consequently, for each positive 
$\varepsilon$, the following convergence takes place: 
\eq{
 \lim_{J\to \infty}\limsup_{n\to \infty}\mu
 \ens{\sup_{j\geqslant J}2^{\alpha j}\max_{r\in D_j}\abs{\lambda_r\left(S_n^{\mathrm{pl}}(m_R)\right)}
 >\varepsilon n^{1/2}}=0.
}
Using Proposition~\ref{generalized_Doob}, we derive the following bound, valid
for each $\varepsilon$ and each $R$,
\begin{multline}\label{bound_truncation}
 \lim_{J\to \infty}\limsup_{n\to \infty}\mu
 \ens{\sup_{j\geqslant J}2^{\alpha j}\max_{r\in D_j}\abs{\lambda_r\left(S_n^{\mathrm{pl}}(m)\right)}
 >\varepsilon n^{1/2}}\leqslant \\
 \leqslant C_p\varepsilon^{-p}\left(\sup_{t\geqslant 0}t^p\mu\ens{\abs m\chi\ens{\abs{m}
 >R}>t}+\sup_{t\geqslant 0}t^p\mu\ens{\mathbb E[\abs m\chi\ens{\abs{m}
 >R}\mid T\mathcal M]>t}\right)+\\
 +\varepsilon^{-p}C_p\mathbb E\left(\left(\mathbb E[m^2\chi\ens{\abs m>R}\mid T\mathcal M]\right)^{p/2}\right).
\end{multline}
The first term is $\sup_{t\geqslant R}t^p\mu\ens{\abs m>t}$, which goes to $0$ 
as $R$ goes to infinity. 

The second term can be bounded by 
$\sup_{t\geqslant R}t^p\mu\ens{\mathbb E[\abs m\mid T\mathcal M]>t}$. Indeed, 
if $t \geqslant R$, we use the inclusion 
\begin{equation}
 \ens{\mathbb E[\abs m\chi\ens{\abs{m}
 >R}\mid T\mathcal M]>t} \subset\ens{\mathbb E[\abs m\mid T\mathcal M]>t},
\end{equation}
and if $t<R$, then accounting the fact that the random variable 
$\mathbb E[\abs m\chi\ens{\abs{m}
 >R}\mid T\mathcal M]$ is greater than $R$, we get 
\begin{align}
 \mathbb E[\abs m\chi\ens{\abs{m}
 >R}\mid T\mathcal M]
 &=\mathbb E[\abs m\chi\ens{\abs{m}\nonumber
 >R}\mid T\mathcal M] \chi \ens{\mathbb E[\abs m\mid T\mathcal M]>R}\\
 &\leqslant \mathbb E[\abs m\mid T\mathcal M] \chi \ens{\mathbb E[\abs m\mid T\mathcal M]>R},
\end{align}
from which it follows that 
\begin{equation}
 t^p\mu\ens{\mathbb E[\abs m\chi\ens{\abs{m}
 >R}\mid T\mathcal M]>t}
 \leqslant R^p\mu\ens{\mathbb E[\abs m\mid T\mathcal M]>R}.
\end{equation}

By Lemma~\ref{tail_cond_expect}, the convergence 
\begin{equation} 
 \lim_{R\to  \infty}\sup_{t\geqslant R}t^p\mu\ens{\mathbb E[\abs m\mid T\mathcal M]>t}=0
\end{equation}
takes place. 

The third term of \eqref{bound_truncation} converges to $0$ as $Réé$ 
goes to infinity by monotone convergence.

This concludes the proof of tightness in Theorem~\ref{sufficient_cond_mart}.

\end{proof}

\subsection{Proof of Theorem~\ref{Hannan_cond}}

By \eqref{regularity}, the equality $f=\sum_{i\geqslant 0}P_i(f)$ holds 
almost surely.
For a fixed integer $K$, we define 
$f_K:=\sum_{i=0}^KP_i(f)$. Then $f_K$ satisfies the conditions of Corollary~\ref{martingale_coboundary}. 

Indeed, we have the equalities
\begin{align}
 P_i(f)-P_0(U^if)&= \mathbb E [f\mid T^i\mathcal M]-\mathbb E [U^if\mid \mathcal M]
 -\mathbb E [f\mid T^{i+1}\mathcal M]+\mathbb E [U^if\mid T\mathcal M]\\
 &=(I-U^i)\mathbb E [f\mid T^i\mathcal M]-(I-U^i)\mathbb E [f\mid T^{i+1}\mathcal M]
\end{align}
and the later term can be expressed as a coboundary noticing that $(I-U^i)=
(I-U)\sum_{k=0}^{i-1}U^k$. Since $P_i(f)$ belongs to the $\mathbb L^p$ space, 
we may write $f_K-\sum_{i=0}^KP_0(U^if)$ as $(I-U)g_K$ where $g_K$ belongs to 
the $\mathbb L^p$ space.
 Defining $m_K:=\sum_{i=0}^KP_0(U^i(f))$, the 
sequence $(m_K\circ T^i)_{i\geqslant 0}$ is a martingale difference sequence 
hence for each positive $\varepsilon$,
\begin{equation}\label{tightness_f_K}
 \lim_{J\to \infty}\limsup_{n\to \infty}\mu
 \ens{\sup_{j\geqslant J}2^{\alpha j}\max_{r\in D_j}\abs{\lambda_r\left(S_n^{\mathrm{pl}}(f_K)\right)}
 >\varepsilon n^{1/2}}=0.
\end{equation}
Now, we have to show that the convergence in \eqref{tightness_f_K} holds if 
$f_K$ is replaced by $f-f_K$. To this aim, we use the inclusion 
\begin{multline}
 \ens{\sup_{j\geqslant J}2^{\alpha j}\max_{r\in D_j}\abs{\lambda_r\left(S_n^{\mathrm{pl}}
 (f-f_K)\right)}
 >\varepsilon n^{1/2}}\subseteq \\
\subseteq \ens{\sup_{j\geqslant 1}2^{\alpha j}\max_{r\in D_j}\abs{\lambda_r
\left(S_n^{\mathrm{pl}}(f-f_K)\right)}
 >\varepsilon n^{1/2}},
\end{multline}
hence 
\begin{align}
 \mu\ens{\sup_{j\geqslant J}2^{\alpha j}\max_{r\in D_j}\abs{\lambda_r\left(S_n^{\mathrm{pl}}
 (f-f_K)\right)}
 >\varepsilon n^{1/2}}&\leqslant 
 \varepsilon^{-p}\norm{\norm{\frac 1{\sqrt n}S_n^{\mathrm{pl}}(f-f_K)}_{\mathcal H_{1/2-1/p}^o}}
 _{p,\infty}^p\\
 &=\varepsilon^{-p}\norm{\norm{\frac 1{\sqrt n}S_n^{\mathrm{pl}}\left(
 \sum_{i\geqslant K+1}P_i(f)\right)}_{\mathcal H_{1/2-1/p}^o}}_{p,\infty}^p,
\end{align}
from which it follows that
\begin{multline}\label{estimation_remainder_martingale}
 \mu\ens{\sup_{j\geqslant J}2^{\alpha j}\max_{r\in D_j}\abs{\lambda_r\left(S_n^{\mathrm{pl}}
 (f-f_K)\right)}
 >\varepsilon n^{1/2}}\\
 \leqslant \varepsilon^{-p}\left(\sum_{i\geqslant K+1}\norm{\norm{\frac 1{\sqrt n}
 S_n^{\mathrm{pl}}\left(
 P_i(f)\right)}_{\mathcal H_{1/2-1/p}^o}}_{p,\infty}\right)^p.
 \end{multline}

Notice that for a fixed $i$, the sequence $(U^l(P_i(f)))_{l\geqslant 1}$ is a 
martingale difference sequence 
(with respect to the filtration $(T^{-i-l}\mathcal M)_{l\geqslant 0}$). Therefore, by 
Proposition~\ref{generalized_Doob}, we obtain 
\begin{equation}
 \norm{\norm{\frac 1{\sqrt n}S_n^{\mathrm{pl}}\left(
 P_i(f)\right)}_{\mathcal H_{1/2-1/p}^o}}_{p,\infty}
 \leqslant C_p\norm{P_i(f)}_p.
\end{equation}

Plugging this estimate into \eqref{estimation_remainder_martingale}, we obtain 
that for some constant $C$ depending only on $p$, 
\begin{equation}\label{tightness_f-f_K}
 \mu\ens{\sup_{j\geqslant J}2^{\alpha j}\max_{r\in D_j}\abs{\lambda_r
 \left(S_n^{\mathrm{pl}}(f-f_K)\right)}
 >\varepsilon n^{1/2}}\leqslant C\varepsilon^{-p}
 \left(\sum_{i\geqslant K+1}\norm{P_i(f)}_p\right)^p.
\end{equation}
Combining \eqref{tightness_f_K} and \eqref{tightness_f-f_K}, we obtain for each $K$:
\begin{multline}
 \lim_{J\to\infty}\limsup_{n\to\infty}
\mu\ens{\sup_{j\geqslant J}2^{\alpha j}\max_{r\in D_j}\abs{\lambda_r
\left(S_n^{\mathrm{pl}}(f)\right)}
>n^{1/2}\varepsilon}\leqslant \\
\leqslant C\varepsilon^{-p}
 \left(\sum_{i\geqslant K+1}\norm{P_i(f)}_p\right)^p.
\end{multline}
Since $K$ is arbitrary, we conclude the proof of Theorem~\ref{Hannan_cond} 
thanks to assumption \eqref{Hannan_condition}.

\bigskip

\textbf{Acknowledgements.} 
The author is grateful to the referee for many comments 
which improved the readability of the paper.

The author would like to thank Dalibor Voln\'y for many useful 
discussions which lead to the counter-example in Theorem~\ref{counter_example}, and also 
Alfredas Ra\v{c}kauskas and Charles Suquet for their support. 

\def\polhk\#1{\setbox0=\hbox{\#1}{{\o}oalign{\hidewidth
  \lower1.5ex\hbox{`}\hidewidth\crcr\unhbox0}}}\def\cprime{$'$}
\providecommand{\bysame}{\leavevmode\hbox to3em{\hrulefill}\thinspace}
\providecommand{\MR}{\relax\ifhmode\unskip\space\fi MR }
\providecommand{\MRhref}[2]{%
  \href{http://www.ams.org/mathscinet-getitem?mr=#1}{#2}
}
\providecommand{\href}[2]{#2}


\begin{thebibliography}{DMV07}

\bibitem[Bil68]{MR1324786}
Patrick Billingsley, \emph{Probability and measure}, third ed., Wiley Series in
  Probability and Mathematical Statistics, John Wiley \& Sons, Inc., New York,
  1995, A Wiley-Interscience Publication. \MR{1324786 (95k:60001)}

\bibitem[Cie60]{MR0132389}
Z.~Ciesielski, \emph{{On the isomorphisms of the spaces {$H\sb{\alpha }$} and
  {$m$}}}, Bull. Acad. Polon. Sci. S{\'e}r. Sci. Math. Astronom. Phys.
  \textbf{8} (1960), 217--222. \MR{0132389 (24 \#A2234)}

\bibitem[DMV07]{MR2359065}
J{\'e}r{\^o}me Dedecker, Florence Merlev{\`e}de, and Dalibor Voln{\'y},
  \emph{{On the weak invariance principle for non-adapted sequences under
  projective criteria}}, J. Theoret. Probab. \textbf{20} (2007), no.~4,
  971--1004. \MR{2359065 (2008g:60088)}

\bibitem[Don51]{MR0040613}
Monroe~D. Donsker, \emph{An invariance principle for certain probability limit
  theorems}, Mem. Amer. Math. Soc., \textbf{1951} (1951), no.~6, 12.
  \MR{0040613 (12,723a)}

\bibitem[Han73]{MR0331683}
E.~J. Hannan, \emph{Central limit theorems for time series regression}, Z.
  Wahrscheinlichkeitstheorie und Verw. Gebiete \textbf{26} (1973), 157--170.
  \MR{0331683 (48 \#10015)}

\bibitem[KR91]{MR1108512}
G{\'e}rard Kerkyacharian and Bernard Roynette, \emph{Une d\'emonstration simple
  des th\'eor\`emes de {K}olmogorov, {D}onsker et {I}to-{N}isio}, C. R. Acad.
  Sci. Paris S\'er. I Math. \textbf{312} (1991), no.~11, 877--882. \MR{1108512
  (92g:60009)}

\bibitem[LV01]{MR1856684}
Emmanuel Lesigne and Dalibor Voln{\'y}, \emph{{Large deviations for
  martingales}}, Stochastic Process. Appl. \textbf{96} (2001), no.~1, 143--159.
  \MR{1856684 (2002k:60080)}

\bibitem[MPU06]{MR2206313}
Florence Merlev{\`e}de, Magda Peligrad, and Sergey Utev, \emph{{Recent advances
  in invariance principles for stationary sequences}}, Probab. Surv. \textbf{3}
  (2006), 1--36. \MR{2206313 (2007a:60025)}

\bibitem[MRS12]{MR3020943}
 J. Markevi{\v{c}}i{\=u}t{\.e}, A. Ra{\v{c}}kauskas, and 
 Ch. Suquet. \emph{{Functional central limit theorems for sums of nearly nonstationary
   processes}}, Lith. Math. J. \textbf{52(3)} (2012), 282--296


\bibitem[Nag03]{MR2021875}
S.~V. Nagaev, \emph{On probability and moment inequalities for supermartingales
  and martingales}, Proceedings of the {E}ighth {V}ilnius {C}onference on
  {P}robability {T}heory and {M}athematical {S}tatistics, {P}art {II} (2002),
  vol.~79, 2003, pp.~35--46. \MR{2021875 (2005f:60098)}

\bibitem[RS03]{MR2000642}
Alfredas Ra{\v{c}}kauskas and Charles Suquet, \emph{Necessary and sufficient
  condition for the {L}amperti invariance principle}, Teor. \u Imov\=\i r. Mat.
  Stat. (2003), no.~68, 115--124. \MR{2000642 (2004g:60050)}

\bibitem[Ste61]{MR0131517}
E.~M. Stein, \emph{On the maximal ergodic theorem}, Proc. Nat. Acad. Sci.
  U.S.A. \textbf{47} (1961), 1894--1897. \MR{0131517 (24 \#A1367)}

\bibitem[Suq99]{MR1736910}
Ch. Suquet, \emph{{Tightness in {S}chauder decomposable {B}anach spaces}},
  {Proceedings of the {S}t. {P}etersburg {M}athematical {S}ociety, {V}ol. {V}}
  (Providence, RI), {Amer. Math. Soc. Transl. Ser. 2}, vol. 193, Amer. Math.
  Soc., 1999, pp.~201--224. \MR{1736910 (2000k:60009)}

\end{thebibliography}
\end{document}